\documentclass[11pt]{amsart}
\usepackage{amsmath,enumerate,etoolbox,mathrsfs,booktabs,float}

\numberwithin{equation}{section}
\theoremstyle{plain}
\newtheorem{theorem}{Theorem}[section]
\newtheorem{proposition}[theorem]{Proposition}
\newtheorem*{proposition*}{Proposition}

\newtheorem{lemmma*}{Lemma}

\theoremstyle{definition}
\newtheorem{definition}[theorem]{Definition}
\newtheorem*{definition*}{Definition}

\newtheorem{remark*}{Remark}
\newcommand{\GL}{{\rm GL}}
\newcommand\csum[1]{\mathscr #1}
\newcommand{\Reg}{\mathrm{Reg}}
\newcommand{\trace}{\mathrm{Tr}}
\newcommand{\Irr}{\mathrm{Irr}}
\newcommand{\Cl}{\mathrm{Cl}}
\newcommand{\des}{\mathrm{des}}
\newcommand{\CF}{\mathrm{CF}}
\newcommand\ol[1]{\overline{#1}}
\newcommand{\hook}{\eta}
\newcommand{\ld}{\ell}
\newcommand\seq{\mathfrak a}
\newcommand\emptysq{[\,\text{-}\, , \,\text{-}\,]}

\begin{document}
\title{On Foulkes characters}
\author[Alexander~R.~Miller]{Alexander~R.~Miller}
\address{United States of America}
\begin{abstract}
  Orthogonality relations for Foulkes characters
  of full monomial groups are presented, along with three solutions to the
  problem of decomposing products of these characters,
  and new applications, including a product reformulation of
  a Markov chain for adding random numbers studied by Diaconis and Fulman, 
  and a new proof of a theorem of Zagier which generalizes
  one of Harer and Zagier on the enumeration of
  Riemann surfaces of a given genus. 
\end{abstract}
\maketitle
\thispagestyle{empty}
\section{Introduction}
Let $\ell(\pi)$ denote the number of cycles of a permutation $\pi\in S_n$. 
Let $\phi_0,\phi_1,\ldots,\phi_{n-1}$ be the Foulkes characters of $S_n$,
so  
$\phi_i$ is afforded by the sum of Specht modules $V_\beta$ with $\beta$ of border shape with $n$ boxes and $i+1$ rows.
For history and properties, see Chapter 8 of Kerber's book \cite{Kerber}.
Our starting point is the classical fact that the $\phi_i$'s depend only on length in the sense that
\begin{equation}\label{length property}
  \phi_i(\sigma)=\phi_i(\tau)\quad\text{whenever}\quad \ell(\sigma)=\ell(\tau),
\end{equation}
and in fact the $\phi_i$'s form a basis for the space $\CF_\ell(S_n)$ of all class functions $\vartheta$ that depend only on $\ell$, with
each $\vartheta\in \CF_\ell(S_n)$ decomposing uniquely as 
\begin{equation}\label{decomposition Sn}
  \vartheta =\sum_{i=0}^{n-1} \frac{\langle \vartheta,\epsilon_i\rangle}{\epsilon_i(1)}\phi_i,
\end{equation}
where $\epsilon_i$ is the irreducible character $\chi_\lambda$ for the hook shape $\lambda=(n-i,1^i)$, so $\epsilon_i(1)=\binom{n-1}{i}$.
Other important facts about the $\phi_i$'s include: They decompose the character $\rho$ of the regular representation:
\begin{equation}\label{reg decomposition}
  \phi_0+\phi_1+\ldots+\phi_{n-1}=\rho.
\end{equation}
Their degrees are Eulerian numbers:
\begin{equation}\label{Eulerian degrees}
  \phi_i(1)=|\{\pi\in S_n \mid \des(\pi)=i\}|,\quad \des(\pi)=|\{i \mid \pi(i)>\pi(i+1)\}|. 
\end{equation}
They branch according to 
\begin{equation}\label{Sn branching rule}
  \phi_i|_{S_{n-1}}=(n-i)\phi_{i-1}+(i+1)\phi_i.
\end{equation}
And they even admit a closed-form expression:
\begin{equation}\label{Sn expression}
  \phi_i(\pi)=\sum_{j=0}^{n-1} (-1)^{i-j}\binom{n+1}{i-j}(j+1)^{\ell(\pi)}.
\end{equation}
But two questions remain unanswered.
\begin{enumerate}[\it {Question} 1.\ ]
\item\label{Q1} How does a product $\phi_i\phi_j$ decompose into a sum of $\phi_k$'s?
\item\label{Q2} What is the inner product $\emptysq$ with respect to which the $\phi_i$'s form an orthonormal basis?
\end{enumerate}
We answer both of these questions in the next section. 

For decomposing products, we present 3 solutions. The first is a combinatorial solution
which follows from a recent result that interprets the values $\phi_i(\pi)$ as coefficients of Loday's Eulerian idempotents from cyclic homology \cite{Loday} in
certain sums in the group algebra $\mathbb C[S_n]$. 
The second solution is an explicit closed-form solution using
\eqref{Sn expression}. The third solution is perhaps the most surprising, being 
a recursive solution given by Delsarte in 1976 in a context void of characters and groups,
and given 4 years before the $\phi_i$'s were introduced by Foulkes in 1980.
Delsarte's work, which had been overlooked up to now, adds yet another surprising place where Foulkes characters arise.

A few years ago, Diaconis and Fulman connected the $\phi_i$'s with adding random numbers \cite{DF1}.
Denote by $\Phi$ the character table
\[\Phi=\left(\phi_i(C_{n-j})\right)_{0\leq i,j\leq n-1},\]
where
\[C_i=\{\pi\in S_n \mid \ell(\pi)=i\}\]
and for any $\vartheta\in \CF_\ell(S_n)$ we denote by $\vartheta(C_i)$ the value $\vartheta(\pi)$ for any ${\pi\in C_i}$.
Holte \cite{Holte} studied the carries that occur when adding $n$ random numbers in base~$b$, particularly the Markov chain with transition matrix
\begin{equation*}
  M=(M(i,j))_{0\leq i,j\leq n-1}
\end{equation*}
given by 
\begin{equation}\label{carries matrix}
  M(i,j)=\text{chance}\{\text{next carry is $j$} \mid \text{last carry is $i$}\}.
\end{equation}
Diaconis and Fulman found that the transposed columns of $\Phi$ are left eigenvectors, in particular
\[\Phi^tM=D\Phi^t,\]
where $D={\rm diag}(b^{0},b^{-1},\ldots,b^{-n+1})$.

We consider not adding random numbers and keeping track of carries, but multiplying random $n$-cycles in $S_n$ and counting factorizations.
Let $\sigma$ and $\tau$ be $n$-cycles $(i_1\, i_2\, \ldots\, i_n)$ chosen uniformly at random from $C_1$, and consider the expected number of ways
that the product $\sigma\tau$ can be written as a product $\alpha\beta$ with $\alpha\in C_i$ and $\beta\in C_j$, i.e.\
\[{\mathbf E}|\sigma C_i\cap \tau C_j|.\]
Dividing by $n!$ gives a probability distribution on pairs $(C_i,C_j)$, and our answer to Question~\ref{Q2} is that
the $\phi_i$'s form an orthonormal basis with respect to the inner product on $\CF_\ell(S_n)$ defined by 
\[[\vartheta,\psi]=\frac{1}{|S_n|}\sum_{i,j=0}^{n-1} \vartheta(C_i)\ol{\psi(C_j)}{\mathbf E}|\sigma C_i\cap \tau C_j|.\]
As a remarkable consequence, we find that the $\phi_i$'s
arise in a natural way from multiplying random $n$-cycles: 
they result from the inner product $\emptysq$
by applying the Gram--Schmidt process to the natural basis of characters $1^{\ld},2^{\ld},\ldots, n^{\ld}$ in $\CF_{\ell}(S_n)$.
This is analogous to how the irreducible characters of $S_n$ can be obtained by 
taking the usual inner product on class functions of $S_n$, namely
\[\langle \vartheta,\psi\rangle =\frac{1}{|S_n|} \sum_{K\in\Cl(S_n)} \vartheta(K)\ol{\psi(K)} |K|,\]
taking a natural choice of permutation characters indexed by partitions, namely $(1_{S_\lambda})^{S_n}$ with a certain natural order, and applying the Gram--Schmidt process.

As a new application of Foulkes characters,
we give a short proof 
of a celebrated result of Zagier which generalizes one of Harer and Zagier on the enumeration of Riemann surfaces of a given genus.
We also rewrite the Markov chain for carries in terms of our inner product
$[\,\text{-}\, ,\,\text{-}\, ]$ and products of characters in $\CF_\ell(S_n)$:
\[M(i,j)=[\phi_i,b^{\ld-n}\phi_j],\]
which is not generally equal to $M(j,i)$.

In the second part of the paper, Section~\ref{Sect G(r,1,n)}, we
answer Questions~\ref{Q1} and~\ref{Q2} for the full monomial groups
$G(r,1,n)$ with $r>1$.  
The author introduced analogues of Foulkes characters for these groups, as well as many other reflection groups, in \cite{M1}, where they were
constructed from
certain reduced homology groups for subcomplexes of the Milnor fiber complex, 
which is a certain wedge of spheres that is an equivariant strong deformation
retract of a Milnor fiber from the invariant theory of the group,  
and then used various machinery 
to prove, among other things \cite{M1,M2,M3,M4},
analogues of \eqref{length property}--\eqref{Sn expression}. 
The role of $\ell$ is played by
$n-\mathfrak l$, where $\mathfrak l$ is the most natural choice of ``length'',
\[\mathfrak l(x)=\min\{k\geq 0\mid x=y_1y_2\ldots y_k\ \text{for some reflections $y_i\in G(r,1,n)$}\}.\]
In addition to enjoying properties analogous to \eqref{length property}--\eqref{Sn expression}, the Foulkes characters of $G(r,1,n)$ were shown in \cite{M2}
to play the role of irreducibles among the characters of $G(r,1,n)$
that depend only on $\mathfrak l$ in the sense that 
the characters of $G(r,1,n)$ that depend only on $\mathfrak l$ are
precisely the unique non-negative integer linear combinations
of the Foulkes characters. So our answers to Questions~\ref{Q1} and \ref{Q2}
round out a truly remarkable story for the groups $G(r,1,n)$ with $r>1$,
particularly the hyperoctahedral groups $G(2,1,n)$. 

Question~\ref{Q1} for $G(r,1,n)$ has answers that
are similar to our answers for~$S_n$. Question~\ref{Q2} for $G(r,1,n)$ 
is more complicated than for $S_n$, but our answer is of a similar flavor
and simplifies in the case of the hyperoctahedral
group.
As in the case of type A, benefits 
include a probability distribution
on the analogues of the pairs $(C_i,C_j)$, 
and a new construction of the Foulkes characters of
$G(r,1,n)$ 
in terms of multiplying random elements and applying the Gram--Schmidt process  
to a natural basis.
Another application is a remarkable rewriting of 
a Markov chain studied by Diaconis and Fulman for adding random numbers in
balanced ternary, a number system that both reduces carries and,
in the words of Donald Knuth, is ``perhaps the prettiest number system of~all.''

\section{Type A}
\subsection{The inner product}
We start with our answer to Question~\ref{Q2} for $S_n$.
Given a subset $A$ of a group $G$,
we denote by $\csum{A}$ the sum $\sum_{a\in A}a$ in $\mathbb C[G]$.

\begin{definition} For $\vartheta,\psi\in \CF_\ell(S_n)$, and for
  $n$-cycles $\sigma$ and $\tau$ chosen uniformly at random from $C_1$, we define 
\[[\vartheta,\psi]=\frac{1}{|S_n|}\sum_{i,j=1}^n \vartheta(C_i)\overline{\psi(C_j)} \mathbf E |\sigma C_i\cap \tau C_j|.\]
\end{definition}

\begin{proposition}\label{general [,] prop}
  For $\vartheta,\psi\in \CF_\ell(S_n)$,
  \begin{equation}
    [\vartheta,\psi]
    =
    \sum_\chi
    \left\langle \vartheta,\frac{\chi}{\chi(1)}\right\rangle
    \left\langle \overline{\psi},\frac{\chi}{\chi(1)}\right\rangle,
  \end{equation}
  where the sum is over the irreducible characters $\epsilon_i=\chi_{(n-i,1^i)}$, $0\leq i\leq n-1$. 
\end{proposition}

\begin{proof}
 Denoting the regular representation of $S_n$ by $\Reg$, we have
\begin{align*}
  [\vartheta,\psi]
  &=
    \frac{1}{|S_n|^2} \trace\circ\Reg\left(\sum_{i,j=1}^n \vartheta(C_i)\overline{\psi(C_j)}\frac{\csum{C_1}^2\csum{C_i}\csum{C_j}}{|C_1|^2} \right)\\
  &=\frac{1}{|S_n|^2}  \sum_{\chi\in\Irr(S_n)} \sum_{i,j=1}^n \frac{\vartheta(C_i)\overline{\psi(C_j)}}{|C_1|^2} \chi(1)
    \frac{\chi(C_1)^2|C_1|^2}{\chi(1)^2} \sum_{x\in C_i}\frac{\chi(x)}{\chi(1)} \sum_{y\in C_j}\chi(y)\\
  &=\frac{1}{|S_n|^2}  \sum_{\chi\in\Irr(S_n)} \sum_{i,j=1}^n  \chi(C_1)^2
    \sum_{x\in C_i}\frac{\vartheta(C_i) \chi(x)}{\chi(1)} \sum_{y\in C_j}\frac{\overline{\psi(C_j)}\chi(y)}{\chi(1)}\\
  &=\sum_{\chi\in\Irr(S_n)} \chi(C_1)^2 \left\langle \vartheta,\frac{\chi}{\chi(1)}\right\rangle \left\langle \overline{\psi},\frac{\chi}{\chi(1)}\right\rangle \\
  &=\sum_\chi \left\langle \vartheta,\frac{\chi}{\chi(1)}\right\rangle \left\langle \overline{\psi},\frac{\chi}{\chi(1)}\right\rangle,
\end{align*}
where the last sum is over all $\chi_{(n-i,1^i)}$ with $0\leq i\leq n-1$. 
\end{proof}

\begin{theorem}
  The characters $\phi_0,\phi_1,\ldots,\phi_{n-1}$ form an
  orthonormal basis for the Hilbert space $\CF_\ell(S_n)$ with inner product
  $\emptysq$.
\end{theorem}

\begin{proof}
By property \eqref{decomposition Sn} and Proposition~\ref{general [,] prop}.
\end{proof}

A natural choice of basis for $\CF_\ell(S_n)$ that is composed of characters is
$1^{\ld},2^{\ld},\ldots, n^{\ld}$, the character  
$k^{\ld}:\pi\mapsto k^{\ld(\pi)}$ being afforded by $(\mathbb C^k)^{\otimes n}$
with 
\[\pi.(v_1\otimes v_2\otimes\ldots\otimes v_n)=v_{\pi^{-1}(1)}\otimes v_{\pi^{-1}(2)}\otimes \ldots\otimes v_{\pi^{-1}(n)}.\]

\begin{theorem}
  The characters $\phi_0,\phi_1,\ldots,\phi_{n-1}$
  result from the inner product $\emptysq$
  by applying the Gram--Schmidt 
  process to the 
  characters $1^{\ld},2^{\ld},\ldots, n^{\ld}$.
\end{theorem}

\begin{proof}
  By \eqref{Sn expression}, we have  
    $\Phi=LV$ with
  \[
    L=\left((-1)^{i-j}\binom{n+1}{i-j}\right)_{0\leq i,j\leq n-1},\quad
    V=\left((i+1)^{\ld(C_{n-j})}\right)_{0\leq i,j\leq n-1},
  \]
  so $L$ is lower unitriangular and $V$ is the character table of the 
  $k^{\ld}$.
  But this means that the rows of $\Phi$ are obtained by applying the Gram--Schmidt process to
  the rows of $V$ using the inner product with respect to which the
  rows of $\Phi$ are orthonormal. 
\end{proof}

\subsection{}
We remark on a formula for Foulkes characters that is similar to some well-known formulas for various systems of orthogonal polynomials, including Legendre polynomials, Hermite polynomials $(2X-\frac{d}{dX})^n\cdot 1$, and Laguerre polynomials $\frac{1}{n!}(\frac{d}{dX}-1)^nX^n$. It appears in the work of Diaconis and Fulman~\cite{DF1} in a slightly different form.

Let
\[A_n=\sum_{\pi\in S_n} X^{\des(\pi)},\]
so 
\[A_0=1,\quad A_1=1,\quad A_2=1+X,\quad A_3=1+4X+X^2,\quad \ldots,\]
and 
\begin{equation}\label{Eulerian diff}
  \left(1+X\frac{d}{dX}\right)^n\frac{1}{1-X}=\frac{A_n}{(1-X)^{n+1}}.
\end{equation}

\begin{theorem}[Diaconis--Fulman]\label{Diaconis--Fulman formula} For $1\leq j \leq n$,
  \begin{equation}\label{OP type formula}
    \sum_{i=0}^{n-1} \phi_i(C_j) X^i= (1-X)^{n+1}\left(1+X\frac{d}{dX}\right)^j\frac{1}{1-X}.
    \end{equation}
\end{theorem}

\begin{proof}
  Denoting by $\phi_i^{(n)}$ and $C_j^{(n)}$ the $\phi_i$ and $C_j$ for $S_n$, we have \cite{Foulkes,Kerber,M1}
  \begin{equation}\label{row rec}
    \phi^{(n)}_i(C_j^{(n)})=\phi^{(n-1)}_i(C_j^{(n-1)})-\phi^{(n-1)}_{i-1}(C_j^{(n-1)})
  \end{equation}
  for $0\leq i\leq n-1$ and $1\leq j\leq n-1$, where we take $\phi_{-1}^{(n-1)}=\phi_{n-1}^{(n-1)}=0$. 
  So the $S_{n-1}$ cases of \eqref{OP type formula} imply the first $n-1$ cases of \eqref{OP type formula} for $S_n$, while
  equality holds for $j=n$ by \eqref{Eulerian degrees} and \eqref{Eulerian diff}.
\end{proof}

\subsection{Decomposing products of Foulkes characters}
We now present three solutions to computing $[\phi_i\phi_j,\phi_k]$ for $0\leq i,j,k\leq n-1$. 

\subsubsection{First solution}
Our first solution is
a combinatorial solution in terms of descents, 
and it is a corollary of an earlier theorem involving
Loday's Eulerian idempotents \cite{Loday}. 
Writing 
\[\csum{D_i}=\sum_{{\pi\in S_n}\atop{\des(\pi)=i}}\pi,\]
the Eulerian idempotents $\csum{E_0},\csum{E_1},\ldots,\csum{E_{n-1}}\in \mathbb Q[S_n]$ are defined by
\begin{equation}\label{Eulerian Idempotents}
  \sum_{i=0}^{n-1}\binom{X+n-1-i}{n} \csum{D_i}=\sum_{i=0}^{n-1}\csum{E_{n-1-i}}X^{n-i},
\end{equation}
and the following is a special case of Theorem~9 in \cite{M1}. 
\begin{theorem}\label{transition result}
  $\Phi^{t}$ is the transition matrix from
  \[\csum{D_0},\csum{D_1},\ldots, \csum{D_{n-1}}\]
  to
  \[\csum{E_{n-1}},\csum{E_{n-2}},\ldots, \csum{E_0},\]
  so
  \begin{equation}\label{D into E's}
    \csum{D_i}=\sum_{j=0}^{n-1} \phi_i(C_{n-j}) \csum{E_{n-1-j}}
  \end{equation}
  and 
  \begin{equation}\label{Phi inverse}
    \Phi^{-t}=\left(\text{Coeff.\ of $X^{n-j}$ in }\binom{X+n-1-i}{n}\right)_{0\leq i,j\leq n-1}.
  \end{equation}
\end{theorem}

As a consequence of Theorem~\ref{transition result}, we have the following.

\begin{theorem}\label{Sn combinatorial solution} For any fixed $z\in S_n$ with $\des(z)=k$, 
  \begin{equation}\label{Sn combinatorial solution Eq}
    [\phi_i\phi_j,\phi_k]=|\{(x,y)\in S_n^2 \mid \des(x)=i,\ \des(y)=j,\ xy=z\}|.
  \end{equation}
\end{theorem}
\begin{proof}
  The $\csum{D_i}$'s form a basis for a subalgebra of $\mathbb C[S_n]$, and the $\csum{E_i}$'s are orthogonal idempotents,
  so by
  \eqref{D into E's}, $[\phi_i\phi_j,\phi_k]$ is the coefficient of
  $\csum{D_k}$ in $\csum{D_i}\csum{D_j}$. Hence \eqref{Sn combinatorial solution Eq}.
\end{proof}

\subsubsection{Second solution}
Our second solution is a closed-form solution which uses
the decomposition in \eqref{decomposition Sn},  
the explicit expression for $\phi_i(\pi)$ in \eqref{Sn expression},
and the fact that, for any $\chi_\lambda\in\Irr(S_n)$, 
\begin{equation}\label{hook content formula}
\langle X^{\ld},\chi_\lambda\rangle=\prod_{b \in \lambda}\frac{X+c(b)}{h(b)},
\end{equation}
where for a box $b\in\lambda$ located in the $i$-th row and $j$-th column,
\[c(b)=j-i,\quad h(b)=\lambda_i-j+1+|\{k>i \mid \lambda_k\geq j\}|.\]

\begin{theorem}
  \begin{equation}
    [\phi_i\phi_j,\phi_k]
    =
    \sum_{{0\leq u\leq i}\atop{0\leq v\leq j}}(-1)^{i-u}(-1)^{j-v}\binom{n+1}{i-u}\binom{n+1}{j-v}\binom{uv+u+v+n-k}{n}. 
  \end{equation}
\end{theorem}

\begin{proof}
  By \eqref{decomposition Sn} and \eqref{Sn expression},
  \begin{align*}
    [\phi_i\phi_j,\phi_k]
    &=
    \left\langle \phi_i\phi_j,\frac{\epsilon_k}{\epsilon_k(1)}\right\rangle\\
    &= \sum_{u,v=0}^{n-1} (-1)^{i-u}(-1)^{j-v}\binom{n+1}{i-u}\binom{n+1}{j-v}\left\langle\left((u+1)(v+1)\right)^{\ld},\frac{\epsilon_k}{\epsilon_k(1)}\right\rangle,
  \end{align*}
  and by \eqref{hook content formula},
  \[\left\langle\left((u+1)(v+1)\right)^{\ld},\frac{\epsilon_k}{\epsilon_k(1)}\right\rangle
    =\binom{(u+1)(v+1)+n-1-k}{n}.
  \]
\end{proof}

\subsubsection{Third solution} The third solution is a recursive solution due to P.~Delsarte \cite{Delsarte}.
For Foulkes characters $\phi_i,\phi_j,\phi_k$ of
$S_n$, let us write
\[\phi_i\phi_j=\sum_{k=0}^{n-1} c_{ijk}^{(n)} \phi_k.\]
Delsarte defines recursively certain values $F(i,k,n)$, $0\leq i,k\leq n$, that depend on a parameter $q$ and initial
conditions $F(0,k,m)$ with $0\leq k\leq m$, and he considers the matrix $P_{n-1}=(F(i,k,n-1))_{0\leq i, k\leq n-1}$.
Although Delsarte did not specialize in this way, taking
$q=1$ and $F(0,k,m)$ to be the Eulerian number $|\{\pi\in S_{m+1}\mid \des(\pi)=k\}|$, and then comparing Delsarte's definition with
\eqref{Eulerian degrees} and \eqref{row rec}, we find that the transpose of Delsarte's matrix $P_{n-1}$ becomes the 
Foulkes character table $\Phi$ of $S_n$, so that
\begin{equation}\label{Foulkes Delsarte}
  \phi_i(C_{n-j})=F(j,i,n-1),\quad 0\leq i,j\leq n-1.
\end{equation}
In addition to finding very general expressions for the $F(i,k,n)$ and the determinant of $P_{n-1}$ in Theorems 2 and 3 of \cite{Delsarte}, Delsarte
 also
 found a recursive solution for calculating the $c_{ijk}^{(n)}$'s, since
 $c_{ij0}^{(n)}=[\phi_i,\phi_j]=\delta_{ij}$.
\begin{theorem}[Delsarte]
\[c_{i+1,j+1,k+1}^{(n+1)}-c_{i+1,j+1,k}^{(n+1)}=-c_{i,j,k}^{(n)}+c_{i+1,j,k}^{(n)}+c_{i,j+1,k}^{(n)}-c_{i+1,j+1,k}^{(n)}.\]
\end{theorem}

\subsection{}
Before moving on, we give 
another useful consequence of Theorem~\ref{transition result}. Let 
\begin{equation}
  \phi=\phi_0+X\phi_1+X^2\phi_2+\ldots +X^{n-1}\phi_{n-1}.
\end{equation}

\begin{theorem}\label{General inversion}
  For any two sequences $a_1,a_2,\ldots,a_n$ and $b_1,b_2,\ldots,b_n$, 
  \[\sum_{i=1}^n a_i X^i =\sum_{k=1}^n b_k \binom{X+n-k}{n}\]
  if and only if
  \[\sum_{i=1}^n a_i\phi(C_i)  =\sum_{k=1}^n b_k X^{k-1}.\]
\end{theorem}

\begin{proof}
  Let
  \[a=\begin{pmatrix} a_n \\ a_{n-1}\\ \vdots \\ a_1\end{pmatrix},\quad
    b=\begin{pmatrix} b_1 \\ b_2\\ \vdots \\ b_n\end{pmatrix},\quad
    x=\begin{pmatrix} X^n \\ X^{n-1}\\ \vdots \\ X^1\end{pmatrix},\quad
    y=\begin{pmatrix} \binom{X+n-1}{n} \\ \binom{X+n-2}{n}\\ \vdots \\\binom{X}{n} \end{pmatrix}.\]
  Then \eqref{Phi inverse} can be rewritten as
  \[\Phi^{-t}x=y.\]
  So  
  \[a^tx=b^t y\quad \Leftrightarrow\quad a^tx=b^t \Phi^{-t}x\quad \Leftrightarrow\quad a^t \Phi^tx=b^t x\quad \Leftrightarrow\quad a^t\Phi^{t}z=b^tz,\]
  where $z=(1,X,\ldots, X^{n-1})^t$ is obtained from $x$ by replacing $X$ by $X^{-1}$ and then multiplying by $X^{n}$.
\end{proof}

\subsection{Zagier's result}
As an application of Foulkes characters, particularly
the formula of Diaconis and Fulman in Theorem~\ref{Diaconis--Fulman formula}
and our inversion result in Theorem~\ref{General inversion}, we
give a new short proof of a well-known result of Zagier~\cite{Zagier}.

\begin{theorem}[Main theorem of Zagier]
  For any conjugacy class ${K\in \Cl(S_n)}$ and any $n$-cycle $\sigma\in S_n$, let
\[p_i(K)=\frac{|\{\tau\in K\mid \tau\sigma \text{ \rm has $i$ cycles}\}|}{|K|}.\]
  The numbers $p_i(K)$ are determined by
  \begin{equation}\sum_{i=1}^n p_i(K)P_i(X)=\frac{\wp(K,X)}{(1-X)^{n+2}},\end{equation}
  where $\wp(K,X)=\det(1-\tau X,\mathbb C^n)$ is the characteristic polynomial of an element $\tau\in K$
  under the permutation representation $\tau\mapsto (\delta_{i\tau(j)})_{i,j}$ and
  \[P_1(X)=\frac{1}{(1-X)^2},\quad P_2(X)=\frac{1+X}{(1-X)^3},\quad P_3(X)=\frac{1+4X+X^2}{(1-X)^4},\quad\ldots\]
  are the polynomials in $\frac{1}{1-X}$ defined by $P_i(X)=\sum_{k=1}^\infty k^i X^{k-1}\in\mathbb Z[[X]]$.
\end{theorem}

\begin{proof}
  Writing $L(X)=\sum_{\pi\in S_n} \pi X^{\ell(\pi)}$, and denoting the regular representation by $\Reg$, we have 
  \begin{align*}
    \sum_{i=1}^n p_i(K)X^i
    &=
    \frac{1}{|S_n|}\frac{1}{|K|}\trace\circ\Reg (K\sigma L(X))\\
    &=\frac{1}{|S_n|}\frac{1}{|K|}\sum_{\chi\in\Irr(S_n)} \chi(1)\frac{\chi(K)|K|}{\chi(1)} \chi(\sigma) \sum_{\pi\in S_n} \frac{\chi(\pi) X^{\ell(\pi)}}{\chi(1)}\\
    &=\sum_{k=0}^{n-1}(-1)^k \epsilon_k(K) \left\langle \frac{\epsilon_k}{\epsilon_k(1)},X^{\ld}\right\rangle\\
    &=\sum_{k=0}^{n-1}(-1)^k \epsilon_k(K) \binom{X+n-k-1}{n}.
  \end{align*} 
Equivalently, by Theorem~\ref{General inversion}, 
\begin{equation}\label{Foulkes Zagier}
\sum_{i=1}^n p_i(K)\phi(C_i)=\sum_{k=0}^{n-1}(-1)^k \epsilon_k(K)X^k=\frac{\wp(K,X)}{1-X}.
\end{equation}
By Theorem~\ref{Diaconis--Fulman formula},  \eqref{Foulkes Zagier} is equivalent to 
\begin{equation}
\sum_{i=1}^n p_i(K)P_i(X)=\frac{\wp(K,X)}{(1-X)^{n+2}}.
\end{equation}
\end{proof}

\subsection{Carries in terms of products}
As another application of our framework for Foulkes characters, we give an interesting reformulation of the
Markov chain studied by Holte \cite{Holte} and Diaconis and Fulman \cite{DF1} in terms of our inner product $\emptysq$ and products of characters in $\CF_\ell(S_n)$. 
Let $M$ be the transition matrix given in \eqref{carries matrix}.  

\begin{theorem}\label{Sn rewrite}
  $M(i,j)=[\phi_i,b^{\ell-n}\phi_j]$.
\end{theorem}

\begin{proof}
  Diaconis and Fulman showed that, 
  for $0\leq j\leq n-1$, the row vector $(\phi_0(C_{n-j}),\phi_1(C_{n-j}),\ldots,\phi_{n-1}(C_{n-j}))$ is a left eigenvector of
  the transition matrix $M$  with eigenvalue $b^{-j}$, so
  \begin{equation}\label{Sn eigen}
    \Phi^tM=D\Phi^t,
  \end{equation}
  where $D$ is the diagonal matrix ${\rm diag}(1,b^{-1},b^{-2},\ldots, b^{-n+1})$.

  For $0\leq i,j\leq n-1$, let 
  \[
    \alpha_{ij}=
    \frac{1}{|S_n|}
    \sum_{\pi\in C_{n-j}} \frac{\epsilon_i(\pi)}{\epsilon_i(1)},\qquad
    \epsilon_i=\chi_{(n-i,1^i)},
  \]
  and let $\Lambda$ be the matrix 
  \[
    \Lambda=(\alpha_{ij})_{0\leq i,j\leq n-1}.
  \]
  Then, using \eqref{decomposition Sn},
  \begin{equation}\label{Sn Lambda inverse}
    \Lambda\Phi^t
    =
    \left(\left\langle \frac{\epsilon_i}{\epsilon_i(1)},\phi_j\right\rangle\right)_{0\leq i,j\leq n-1}
    =I.
  \end{equation}
  Hence, by \eqref{Sn eigen} and \eqref{Sn Lambda inverse}, 
  \[
    M=\Lambda D\Phi^t
    =
    \left(\left\langle \frac{\epsilon_i}{\epsilon_i(1)}b^{\ell-n},\phi_j\right\rangle\right)_{0\leq i,j\leq n-1}
    =
  \left([ \phi_i,b^{\ell-n}\phi_j]\right)_{0\leq i,j\leq n-1}.
  \]
\end{proof}

\section{Type B and the other full monomial groups}\label{Sect G(r,1,n)}
We begin by fixing an integer $r>1$,
a primitive $r$-th root of unity $\zeta$,
the cyclic group $Z=\langle \zeta\rangle$,
and a full monomial group
\[G_n=G(r,1,n),\]
so the elements of $G_n$ are the $n$-by-$n$ matrices
$x$ with exactly one nonzero entry in each row and each column,
and with $r$-th roots of unity for the nonzero entries.
Equivalently, the elements 
$x\in G_n$ are the products 
\[x=D.A_{\pi}\]
where $D$ is a diagonal matrix ${\rm diag}(\xi_1,\xi_2,\ldots,\xi_n)$
with $\xi_i\in Z$, and
$A_{\pi}=(\delta_{i\pi(j)})_{1\leq i,j\leq n}$
is the usual matrix of a permutation $\pi\in S_n$.
By the {\it type of~$x$} we shall mean the partition-valued function
\[\lambda: \Cl(Z)\to \mathscr P\]
which takes $\{\zeta^j\}$, $0\leq j\leq r-1$, to
the partition $\lambda^j$ whose parts are the periods
of the cycles $(i_1\, i_2\, \ldots\, i_k)$ of $\pi$
such that $x_{i_1 i_2}x_{i_2 i_3}\ldots x_{i_k i_1}=\zeta^j$,
so two elements of $G_n$ belong to the same conjugacy class if and only
if they have the same type.
We shall denote
by $K_\lambda$ the class of elements of type $\lambda$.
Identifying $\lambda$ with the $r$-tuple of partitions $\lambda^i$,
we shall write
\[\lambda=(\lambda^0,\lambda^1,\ldots,\lambda^{r-1})\in\mathscr P^r\]
and
\[\|\lambda\|=\sum_{i=0}^{r-1} |\lambda^i|=n.\]
In general, for any partition-valued function $f$ on a finite
set $\mathscr S$, we  write
\[\|f\|=\sum_{s\in\mathscr S} |f(s)|.\]

There is also the natural
bijection \cite{Macdonald} between irreducible characters $\chi_\lambda$ of $G_n$ and
partition-valued functions
\[\lambda:\Irr(Z)\to \mathscr P\]
with $\|\lambda\|=n$. 
Denoting by $\varphi_k$ the irreducible character of $Z$ given by
\[\varphi_k(\zeta^s)=\zeta^{ks},\]
and identifying $\lambda$ with the $r$-tuple of values $\lambda(\varphi_i)$,
we shall write
\[\lambda=(\lambda(\varphi_0),\lambda(\varphi_1),\ldots,\lambda(\varphi_{r-1})).\]

With $G_n$ being a reflection group,
there is the natural length function 
\[\mathfrak l(x)=\min \{k\geq 0 \mid x=y_1y_2\ldots y_k\ \text{for some reflections $y_i\in G_n$}\}.\]
For our purposes, we will instead work with another length function $\ell$.
We define, for $x\in G_n$ of type
$\lambda=(\lambda^0,\lambda^1,\ldots,\lambda^{r-1})\in\mathscr P^r$, 
\[\ell(x)=\text{number of parts of $\lambda^0$}.\]
By Proposition 2 of \cite{M2},
\begin{equation}
  \ell(x)=n-\mathfrak l(x)=\dim\ker(x-1),
\end{equation}
so studying $\ell$ is equivalent to studying $\mathfrak l$.
In particular, a function $f$ depends only on $\ell$, in the sense that
$f(x)=f(y)$ whenever $\ell(x)=\ell(y)$, 
if and only if $f$ depends only on $\mathfrak l$. 

The Foulkes characters of $G_n$ were introduced in \cite{M1}, where they
were constructed from certain reduced homology groups
coming from the associated Milnor fiber complex,
which is a certain wedge of spheres that is 
a strong deformation retract
of a Milnor fiber $f_1^{-1}(1)$
coming from the invariant theory of $G_n$. They
are denoted
\[\phi_0,\phi_1,\ldots,\phi_n,\]
and they were shown in \cite{M1} to have some remarkable properties
that are analogous to the type A properties stated in \eqref{length property}--\eqref{Sn expression}.

The $\phi_i$'s form a basis for the space $\CF_\ell(G_n)$ of all
class functions $\vartheta$ that depend only on $\ell$, with
each $\vartheta\in \CF_\ell(G_n)$ decomposing uniquely as 
\begin{equation}\label{decomposition G(r,1,n)}
  \vartheta=\sum_{i=0}^{n} \frac{\langle \vartheta,\epsilon_i\rangle}{\epsilon_i(1)}\phi_i,
\end{equation}
where $\epsilon_i$ is the irreducible character $\chi_{((n-i),(1^i),\emptyset,\emptyset,\ldots,\emptyset)}$,
so $\epsilon_i(1)=\binom{n}{i}$. 
They decompose the character $\rho$ of the regular representation:
\begin{equation}\label{G(r,1,n) reg decomposition}
  \phi_0+\phi_1+\ldots+\phi_{n}=\rho.
\end{equation}
Their degrees are the natural analogues of Eulerian numbers
given by Steingr\'imsson's notion of
descent:
\begin{equation}\label{G(r,1,n) Eulerian degrees}
  \phi_i(1)=|\{x\in G_n \mid \des(x)=i\}|.
\end{equation}
They branch according to 
\begin{equation}\label{G(r,1,n) branching rule}
  \phi_i|_{G_{n-1}}=((n+1)r-(ri+1))\phi_{i-1}+(ri+1)\phi_i.
\end{equation}
And they admit closed-form expressions:
\begin{equation}\label{G(r,1,n) expression}
  \phi_i(x)=\sum_{j=0}^{n} (-1)^{i-j}\binom{n+1}{i-j}(rj+1)^{\ell(x)}.
\end{equation}

We shall denote by $\Phi$ the character table
\[\Phi=(\phi_i(C_{n-j}))_{0\leq i,j\leq n},\]
where, for $0\leq i\leq n$, 
\[C_i=\{x\in G_n \mid \ell(x)=i\},\]
and for any $\vartheta\in\CF_\ell(G_n)$ we denote by
$\vartheta(C_i)$ the value $\vartheta(x)$ for any $x\in C_i$.

\subsection{Fourier transform of $X^{\ld}$}
The Fourier transform of the class function $X^\ld: x\mapsto X^{\ld(x)}$ will play
an important role in what follows.
Given $\lambda\in\mathscr{P}^r$, by $b\in \lambda$ we
shall mean a box $b$ contained in the Young diagram of some $\lambda^j$,
and by $c(b)$ and $h(b)$ we shall mean the usual content and hook-length
associated to the box $b$ in $\lambda^j$.
Given $\lambda\in\mathscr{P}^r$ and $b\in\lambda$, we define
\[
  \delta_0(b)=
  \begin{cases}
    1 & \text{if $b\in\lambda^0$,}\\
    0 & \text{otherwise.}
  \end{cases}
\]

\begin{theorem}\label{r hook content}
  For any $\lambda\in\mathscr{P}^r$ with $\|\lambda\|=n$,
  \begin{equation}\label{r hook chi}
    \langle X^{\ld},\chi_{\lambda}\rangle
    =
    \prod_{b\in\lambda}\frac{\frac{X-1}{r}+c(b)+\delta_0(b)}{h(b)},
  \end{equation}
  \begin{equation}\label{r hook normalized}
    \left\langle X^{\ld},
    \frac{\chi_{\lambda}}{\chi_{\lambda}(1)}\right\rangle
    =
    \frac{1}{n!}
    \prod_{b\in\lambda}\left(\frac{X-1}{r}+c(b)+\delta_0(b)\right),
  \end{equation}
  and, for $\mathfrak X_\lambda:G_n\to\GL_d(\mathbb C)$ affording $\chi_\lambda$,
  \begin{equation}\label{r hook Fourier}
    \sum_{x\in G_n}X^{\ell(x)}\mathfrak X_{\lambda}(x)
    =
    \prod_{b\in\lambda}(X+rc(b)+r\delta_0(b)-1)
    \mathfrak X_{\lambda}(1).
  \end{equation}
\end{theorem}

\begin{proof}
  For any non-negative integer $k$, define
  \[\chi_{n,k}(x)=(kr+1)^{\ell(x)},\quad x\in G_n.\]
  By Proposition 6 and Proposition 7 of \cite{M2}, in the standard notation, see~\cite{M2}, we have 
  \begin{equation}
    \sum_{n\geq 0}{\rm ch}(\chi_{n,k})X^n
    =
    H(\varphi_0)^{k+1}\prod_{j=1}^{r-1}H(\varphi_j)^k.
  \end{equation}
  For any $\varphi\in\Irr(Z)$, by \cite[p.\ 66]{Macdonald}, we have
  \begin{equation}
    H(\varphi)^k=\sum_{\mu\in \mathscr{P}} a_\mu s_\mu(\varphi) X^{|\mu|},
  \end{equation}
  where
  \begin{equation}
    a_\mu=\prod_{b\in \mu}\frac{k+c(b)}{h(b)}.
  \end{equation}
  Hence
  \begin{equation}
    \sum_{n\geq 0}{\rm ch}(\chi_{n,k})X^n
    =
    \sum_{\nu\in\mathscr P^r}a_{\nu}S_{\nu}X^{|\nu|},
  \end{equation}
  where
  \begin{equation}
    a_{\nu}=\prod_{b\in\nu}\frac{k+c(b)+\delta_0(b)}{h(b)}.
  \end{equation}
  Equivalently, for any $\nu\in\mathscr P^r$ with $\|\nu\|=n$,
   \begin{equation}
    \left\langle (kr+1)^{\ld},\chi_{\nu}\right\rangle =\prod_{b\in\nu}\frac{k+c(b)+\delta_0(b)}{h(b)}.
  \end{equation}
  This holds for all non-negative integers $k$, so it holds as an equality of polynomials in $\mathbb C[k]$, and upon replacing $k$ by $\frac{X-1}{r}$, we get \eqref{r hook chi}.
  
  The equality in \eqref{r hook normalized} follows from \eqref{r hook chi}, since 
  \begin{equation}
    \chi_{\lambda}(1)
    =
    \binom{n}{|\lambda^0|,|\lambda^1|,\ldots,|\lambda^{r-1}|}
    \prod_{i=0}^{r-1}\chi_{\lambda^i}(1)
    =
    n!\prod_{i=0}^{r-1}\frac{\chi_{\lambda^i}(1)}{|\lambda^i|!}
    =
    \frac{n!}{\prod_{b\in\lambda} h(b)}.
  \end{equation}
  
  For \eqref{r hook Fourier}, let $L(X)=\sum_{x\in G_n}xX^{\ell(x)}$. 
  $L(X)$ is central, so
  \begin{equation}\label{get alpha}
    \mathfrak X_{\lambda}(L(X))=\alpha\mathfrak X_{\lambda}(1)
  \end{equation}
  for some polynomial $\alpha$. 
  Taking the trace on both sides of \eqref{get alpha}
  and dividing by $\chi_{\lambda}(1)$ gives 
  \begin{equation}\label{alpha first}
    \alpha
    =
    n!r^n\left\langle X^{\ld},\frac{\chi_{\lambda}}{\chi_{\lambda}(1)}\right\rangle.
  \end{equation}
  By \eqref{r hook normalized} and \eqref{alpha first},
    \[\alpha=\prod_{b\in\lambda}\left(X-1+rc(b)+r\delta_0(b)\right).\]
\end{proof}

\subsection{}
There are four important consequences of Theorem \ref{r hook content}.
\subsubsection{}
For  $0\leq k\leq n-1$, we shall write
\[\hook_{s,k}=\chi_{(\emptyset,\ldots,\emptyset,(n-k,1^k),\emptyset,\ldots,\emptyset)},\]
where the hook-shaped partition $(n-k,1^k)$ is in position $0\leq s\leq r-1$.
\begin{proposition}\label{X hook prop}
  \begin{equation}\label{X hook eq}
    \left\langle X^{\ld},\frac{\hook_{s,k}}{\hook_{s,k}(1)}\right\rangle
    =
    \begin{cases}
      \binom{\frac{X-1}{r}+n-k}{n}
      & \text{if $s= 0$,}\\
      \binom{\frac{X-1}{r}+n-k-1}{n}
      & \text{if $s\neq 0$.}\\
    \end{cases}
  \end{equation}
\end{proposition}

\begin{proof}
By \eqref{r hook normalized} of Theorem~\ref{r hook content}.
\end{proof}

\begin{proposition}\label{phi hook prop}
  \begin{equation}\label{phi hook eq}
    \left\langle \phi_i,\frac{\hook_{s,k}}{\hook_{s,k}(1)}\right\rangle
    =
    \begin{cases}
      \delta_{ik}
      & \text{if $s= 0$,}\\
      \delta_{i, k+1}
      & \text{if $s\neq 0$.}
    \end{cases}
  \end{equation}
\end{proposition}

\begin{proof}
  For $0\leq u,v\leq n$, we have \cite[Eq.\ 18]{M2}
  \begin{equation}\label{inverse matrices}
    \sum_{j=0}^n (-1)^{u-j}\binom{n+1}{u-j}\binom{n+j-v}{n}=\delta_{uv}.
  \end{equation}
  See also \cite[Eq.\ 1.6.1]{Loday} and \cite[Eqs.\ 9 and 11]{M1}.
  
By \eqref{G(r,1,n) expression}, \eqref{X hook eq}, and \eqref{inverse matrices}, 
  \begin{align*}
    \left\langle\phi_i,\frac{\hook_{s,k}}{\hook_{s,k}(1)}\right\rangle
    &=\sum_{j=0}^{n} (-1)^{i-j}\binom{n+1}{i-j}\left\langle (rj+1)^{\ld},\frac{\hook_{s,k}}{\hook_{s,k}(1)}\right\rangle\\
    &= \sum_{j=0}^{n} (-1)^{i-j}\binom{n+1}{i-j}\binom{n+j-k-1+\delta_{0s}}{n}\\
    &=\begin{cases}
      \delta_{ik}
      & \text{if $s= 0$,}\\
      \delta_{i,k+1}
      & \text{if $s\neq 0$.}
    \end{cases}
  \end{align*}
\end{proof}

\subsubsection{}
For $0\leq k\leq n$, let 
\begin{equation}
  \epsilon_k=\chi_{((n-k),(1^k),\emptyset,\emptyset,\ldots,\emptyset)}.
\end{equation}

\begin{proposition}\label{X epsilon}
  \begin{equation}\label{X epsilon eq}
    \left\langle X^\ell,\frac{\epsilon_k}{\epsilon_k(1)}\right\rangle
    =\binom{\frac{X-1}{r}+n-k}{n}.
    \end{equation}
  \end{proposition}
  \begin{proof}
    By \eqref{r hook normalized} of Theorem~\ref{r hook content}.
  \end{proof}
  
\begin{proposition}\label{phi epsilon}
  \begin{equation}\label{phi epsilon eq}
    \left\langle \phi_i,\frac{\epsilon_k}{\epsilon_k(1)}\right\rangle
    =\delta_{ik}.
    \end{equation}
\end{proposition}
\begin{proof}
  By \eqref{G(r,1,n) expression}, \eqref{X epsilon eq}, and \eqref{inverse matrices},
  \begin{align*}
    \left\langle\phi_i,\frac{\epsilon_{k}}{\epsilon_{k}(1)}\right\rangle
    &=
    \sum_{j=0}^{n} (-1)^{i-j}\binom{n+1}{i-j}\left\langle (rj+1)^{\ld},\frac{\epsilon_{k}}{\epsilon_{k}(1)}\right\rangle\\
    &=
    \sum_{j=0}^{n} (-1)^{i-j}\binom{n+1}{i-j}\binom{n+j-k}{n}\\
    &=
    \delta_{ik}.
  \end{align*}
\end{proof}

\subsection{Orthogonality relations}
\begin{definition}
  Let $\seq$ be a sequence of classes
  \[K_1,K_2,\ldots, K_m\in\Cl(G_n).\]
  Let $k_i$ be chosen uniformly at random from $K_i$, for $1\leq i\leq m$,
  and consider the expected number of ways that
  the random product $k_1k_2\ldots k_m$ can be written as $ab$ with $a\in C_i$ and $b\in C_j$, i.e.\ 
  \begin{equation}\label{General expectation}
    \mathbf E|k_1k_2\ldots k_m C_i\cap C_j|.
  \end{equation}
  For $\vartheta,\psi\in\CF_\ell(G_n)$, let
  \begin{equation}\label{a sigma inner product}
    [\vartheta,\psi]_{\seq}=\frac{1}{|G_n|}\sum_{i,j=0}^n \vartheta(C_i)\ol{\psi(C_j)}\mathbf E|k_1k_2\ldots k_m C_i\cap C_j|.
  \end{equation}
\end{definition}
Our inner product 
will be a certain convex combination
of $\emptysq_{\seq}$'s.
We shall denote the expectation in \eqref{General expectation} by 
\[\mu_{\seq}(C_i,C_j)=\mathbf E|k_1k_2\ldots k_m C_i\cap C_j|.\]

\begin{proposition}\label{a sigma inner product expression}
  \begin{equation}
    [\vartheta,\psi]_{\seq}=
    \sum_{\chi\in\Irr(G_n)}\frac{\chi(K_1)\chi(K_2)\ldots\chi(K_m)}{\chi(1)^{m-2}}
    \left\langle \vartheta,\frac{\chi}{\chi(1)}\right\rangle
    \left\langle \ol{\psi},\frac{\chi}{\chi(1)}\right\rangle.
  \end{equation}
\end{proposition}

\begin{proof}
  Writing
  \[\csum{K}=\prod_{i=1}^m\frac{\csum{K_i}}{|K_i|},\]
  the right-hand side of \eqref{a sigma inner product} equals
\begin{equation}
  \frac{1}{|G_n|^2}
  \trace\circ \Reg\left(
    \sum_{i,j=0}^n \vartheta(C_i)\ol{\psi(C_j)} \csum{K} \csum{C_i}\csum{C_j}
  \right),
\end{equation}
which in turn equals
\begin{equation}
  \frac{1}{|G_n|^2}
  \sum_{i,j=1}^n
  \sum_{\chi\in\Irr(G_n)}
  \vartheta(C_i)\ol{\psi(C_j)}
  \chi(1)^2
  \omega_\chi(\csum{K}\csum{C_i}\csum{C_j}),
\end{equation}
where $\omega_\chi(W)$ denotes the scalar by which a central element
$W$ of $\mathbb C[G_n]$ acts on a module affording $\chi$, so 
\begin{equation}
  \omega_\chi(\csum{K}\csum{C_i}\csum{C_j})
  =
\left(\prod_{u=1}^m \frac{\chi(K_u)}{\chi(1)}\right)\sum_{x\in C_i}\frac{\chi(x)}{\chi(1)}\sum_{y\in C_j}\frac{\chi(y)}{\chi(1)}.
\end{equation}
Hence
\[[\vartheta,\psi]_{\seq}=
    \sum_{\chi\in\Irr(G_n)}\frac{\chi(K_1)\chi(K_2)\ldots\chi(K_m)}{\chi(1)^{m-2}}
    \left\langle \vartheta,\frac{\chi}{\chi(1)}\right\rangle
    \left\langle \ol{\psi},\frac{\chi}{\chi(1)}\right\rangle. \]
\end{proof}

\begin{proposition}\label{G(r,1,n) orth prop}
  Let $\mathfrak a=(K_1,K_2,\ldots,K_m)$ be a sequence of classes of $G_n$ 
  such that $K_1=K_\lambda$ for some $\lambda$ with $\lambda^s=(n)$ for some $s$. 
  Then 
  \begin{equation}\label{alpha orth}
   [\phi_i,\phi_j]_{\seq}
    =
    \delta_{ij}\xi_\seq(i),
  \end{equation}
  where
  \begin{equation}\label{alpha expression}
    \xi_\seq(i)={\sum_{\chi\in H_i}}\frac{\chi(K_1)\chi(K_2)\ldots \chi(K_m)}{\chi(1)^{m-2}},
  \end{equation}
  and
  \[H_i=\begin{cases}
      \{\hook_{0,0}\} & \text{if $i=0$,}\\
      \{\hook_{0,i}\}\cup\{\hook_{1,i-1},\hook_{2,i-1},\ldots, \hook_{r-1,i-1}\}
      & \text{if $0< i<n$,}\\
      \{\hook_{1,n-1},\hook_{2,n-1},\ldots,\hook_{r-1,n-1}\} &
      \text{if $i=n$.}
    \end{cases}
  \]
\end{proposition}

\begin{proof}
  By Proposition~\ref{a sigma inner product expression},
  the analogue of Murnaghan--Nakayama for $G_n$ given by
  Ariki and Koike \cite{Ariki},  
  and Proposition~\ref{phi hook prop}.
\end{proof}

\begin{definition}\label{mu def}
  Let 
  $\mathfrak a_1,\mathfrak a_2,\ldots,\mathfrak a_5$ be the sequences 
  \begin{align*}
    \seq_1 &=\left(\ K_{((n),\emptyset,\emptyset,\ldots,\emptyset)}\ ,\  K_{((n),\emptyset,\emptyset,\ldots,\emptyset)}\ \right), \\
    \seq_2 &=\left(\ K_{((n),\emptyset,\emptyset,\ldots,\emptyset)}\ ,\  K_{(\emptyset,(n),\emptyset,\ldots,\emptyset)}\ \right),\\
    \seq_3 &=\left(\ K_{((n),\emptyset,\emptyset,\ldots,\emptyset)}\ ,\  K_{((n-1),(1),\emptyset,\ldots,\emptyset)}\ \right),\\
    \seq_4 &=\left(\ K_{((n),\emptyset,\emptyset,\ldots,\emptyset)}\ ,\  K_{((n),\emptyset,\emptyset,\ldots,\emptyset)}\ ,\ K_{((n-1,1),\emptyset,\ldots,\emptyset)}\ ,\ K_{((n-1,1),\emptyset,\ldots,\emptyset)}\ \right),\\
    \seq_5 &=\left(\ K_{((n),\emptyset,\emptyset,\ldots,\emptyset)}\ ,\  K_{((n),\emptyset,\emptyset,\ldots,\emptyset)}\ ,\ K_{((n-1),(1),\emptyset,\ldots,\emptyset)}\ ,\ K_{((n-1,1),\emptyset,\ldots,\emptyset)}\ \right),
  \end{align*}
  and let
  \[\mu_i=\mu_{\seq_i}.\]
  If $n=1$, so $G_n$ is cyclic, let
    \[
    \mu=
      \frac{2}{r}\mu_{1}+
      \frac{r-2}{r}\mu_{2}.
    \]
  If $n\geq 2$, let 
  \begin{equation}\label{mu ge 2}
    \mu=
      \frac{1}{r}\mu_{1}+
      \frac{r-2}{2r}\mu_{2}+
      \frac{1}{4}\mu_{3}+
      \frac{1}{2r}\mu_{4}+
      \frac{r-2}{4r}\mu_{5}.
    \end{equation}
  Define, for $\vartheta,\psi\in\CF_\ell(G_n)$,
  \begin{equation}\label{G(r,1,n) (,) def}
    [\vartheta,\psi]=\frac{1}{|G_n|}\sum_{i,j=0}^n \vartheta(C_i)\ol{\psi(C_j)}\mu(C_i,C_j).
  \end{equation}
  \end{definition}

  It should be noted that the expression for
  $\mu$ given in \eqref{mu ge 2} 
  simplifies in the case of the hyperoctahedral group.
  If $r=2$ and $n\geq 2$, then
  \begin{equation}\label{type B mu}
      \mu
      =
      \frac{1}{2}\mu_1+
      \frac{1}{4}\mu_3+
      \frac{1}{4}\mu_4.
    \end{equation}
  
  As in the case of $S_n$, we have the following properties.
  \begin{proposition} For all $0\leq i,j\leq n$, 
    \begin{equation}
      \mu(C_i,C_j)=\mu(C_j,C_i)\geq 0,
    \end{equation}
    \begin{equation}
      \sum_{i=0}^n \mu(C_i,C_j)=|C_j|,
    \end{equation}
    and
    \begin{equation}
      \sum_{i,j=0}^n\frac{\mu(C_i,C_j)}{|G_n|}=1.
    \end{equation}
  \end{proposition}
  \begin{proof}
    These follow from the definition of $\mu$. 
  \end{proof}
  
  \begin{theorem}
    The characters $\phi_0,\phi_1,\ldots,\phi_n$ form an orthonormal basis for the
    Hilbert space $\CF_\ell(G_n)$ with inner product $\emptysq$.
  \end{theorem}

  \begin{proof}
    The $n=1$ case is a simple calculation, so assume $n\geq 2$.
    By the expression for $\xi_\seq(i)$ in \eqref{alpha expression} and
    the analogue of Murnaghan--Nakayama for $G_n$ given by Ariki and Koike
    \cite{Ariki},
    we have the values in
    Tables~\ref{alpha table}
    and~\ref{alpha table small}
    for~$\xi_{\seq_k}(i)$. 
    \begin{table}[H]
      \centering
      \begin{tabular}{cccccc}
        \toprule
        $\seq$ & $i=0$ & $i=1$ & $2\leq i \leq n-2$ & $i=n-1$ & $i=n$\\ 
        \midrule
        $\seq_1$ & $1$ & $r$ &  $r$  & $r$ & $r-1$\\
        $\seq_2$ & $1$ & $0$ & $0$  & $0$ & $-1$\\
        $\seq_3$ & $1$ & $-1$ & $0$ & $-1$ & $1$\\
        $\seq_4$ & $1$ & $r-1$ & $0$ & $1$ & $r-1$\\
        $\seq_5$ & $1$ & $-1$ & $0$ &  $1$ & $-1$\\
        \bottomrule 
      \end{tabular}\caption{$\xi_{\seq}(i)$ for $n\geq 3$, $\seq=\seq_k$, $1\leq k\leq 5$.}\label{alpha table}
    \end{table}
    \begin{table}[H]
      \centering
      \begin{tabular}{cccccc}
        \toprule
        $\seq$ & $i=0$ & $i=1$ & $i=2$\\ 
        \midrule
        $\seq_1$ & $1$ & $r$  & $r-1$ \\
        $\seq_2$ & $1$ & $0$  & $-1$  \\
        $\seq_3$ & $1$ & $-2$ & $1$   \\
        $\seq_4$ & $1$ & $r$ & $r-1$   \\
        $\seq_5$ & $1$ & $0$ & $-1$   \\
        \bottomrule 
      \end{tabular}\caption{$\xi_{\seq}(i)$ for $n=2$, $\seq=\seq_k$, $1\leq k\leq 5$.}\label{alpha table small}
    \end{table}
  By the definition of $\emptysq$ in \eqref{G(r,1,n) (,) def},
  the orthogonality relation in \eqref{alpha orth},
  and the values in Tables~\ref{alpha table} and~\ref{alpha table small},
  we conclude that the $\phi_i$'s are an orthonormal basis for the Hilbert space
  $\CF_\ell(G_n)$ with inner product $\emptysq$.
  \end{proof}

  A natural choice of basis for  $\CF_{\ell}(G_n)$ that is composed of 
  characters is
  \[1^\ld,(r+1)^\ld,(2r+1)^\ld,\ldots,(nr+1)^\ld.\] 
  For the fact that these are characters, see Proposition 6 in \cite{M2}.
  \begin{theorem}
  The characters $\phi_0,\phi_1,\ldots,\phi_{n}$
  result from the inner product $\emptysq$
  by applying the Gram--Schmidt 
  process to the 
  characters
  \[1^{\ld},(r+1)^{\ld},\ldots, (rn+1)^{\ld}.\]
\end{theorem}

\begin{proof}
  By \eqref{G(r,1,n) expression}, we have  
    $\Phi=LV$ with
  \[
    L=\left((-1)^{i-j}\binom{n+1}{i-j}\right)_{0\leq i,j\leq n},\quad
    V=\left((ri+1)^{\ell(C_{n-j})}\right)_{0\leq i,j\leq n},
  \]
  so the rows of $\Phi$ are obtained by applying the Gram--Schmidt process
  to
  the rows of $V$ using the inner product with respect to which the
  rows of $\Phi$ are orthonormal. 
\end{proof}

We include the following orthogonality relation of independent interest.
  \begin{proposition}
    Let $\sigma$ be an element of some class $K_{\lambda}$ of $G_n$
    with ${\lambda^i=(n)}$ for some $i>0$. 
    Then 
  \begin{equation}
\frac{1}{|G_n|}\sum_{u,v=0}^n\phi_i(C_u)\phi_j(C_v)|\sigma C_u\cap C_v|=(-1)^i\binom{n}{i}\delta_{ij}.
  \end{equation}
\end{proposition}
\begin{proof}
  By Proposition~\ref{G(r,1,n) orth prop} with $m=1$ and $K_1=K_\lambda$,
  and using the analogue of the Murnaghan--Nakayama rule for $G_n$.
\end{proof}

\subsection{Decomposing products of Foulkes characters}
We give two solutions to computing $[\phi_i\phi_j,\phi_k]$. The first is a closed-form solution, and the second is
a combinatorial solution.
\subsubsection{First solution}
\begin{theorem}
  \begin{equation}
    [\phi_i\phi_j,\phi_k]=\sum_{{0\leq u\leq i}\atop{0\leq v\leq j}}(-1)^{i-u}(-1)^{j-v}\binom{n+1}{i-u}\binom{n+1}{j-v}\binom{ruv+u+v+n-k}{n}.
  \end{equation}
\end{theorem}

\begin{proof}
  By \eqref{G(r,1,n) expression} and Proposition~\ref{phi epsilon},
  \begin{align*}
    [\phi_i\phi_j,\phi_k] &= \left\langle \phi_i\phi_j,\frac{\epsilon_k}{\epsilon_k(1)}\right\rangle\\
                          &= \sum_{u,v=0}^n (-1)^{i-u}(-1)^{j-v}\binom{n+1}{i-u}\binom{n+1}{j-v}\left\langle\left((ru+1)(rv+1)\right)^{\ld},\frac{\epsilon_k}{\epsilon_k(1)}\right\rangle,
  \end{align*}
  and by Proposition~\ref{X epsilon},
  \begin{equation}
    \left\langle\left((ru+1)(rv+1)\right)^{\ld},\frac{\epsilon_k}{\epsilon_k(1)}\right\rangle
    =
    \binom{ruv+u+v+n-k}{n}.
  \end{equation}
\end{proof}

\subsubsection{Second solution}
Just as for $S_n$, our combinatorial solution for
computing $[\phi_i\phi_j,\phi_k]$ is
in terms of descents and certain idempotents.
Using Steingr\'imsson's notion of descent for $G_n$ and writing
\[\csum{D_i}=\sum_{{x\in G_n}\atop{\des(x)=i}}x,\]
the Eulerian idempotents $\csum{E_0},\csum{E_1},\ldots,\csum{E_{n}}\in \mathbb Q[G_n]$ for $G_n$ are defined by
\begin{equation}\label{Eulerian Idempotents G(r,1,n)}
  \sum_{i=0}^{n}\binom{n+\frac{X-1}{r}-i}{n} \csum{D_i}=\sum_{j=0}^{n}\csum{E_{n-j}}X^{n-j},
\end{equation}
see \cite{M1} and references therein.  
The following is a special case of \cite[Thm.~9]{M1}. 
\begin{theorem}\label{transition result G(r,1,n)} 
  $\Phi^{t}$ is the transition matrix from
  \[\csum{D_0},\csum{D_1},\ldots, \csum{D_{n}}\]
  to
  \[\csum{E_{n}},\csum{E_{n-1}},\ldots, \csum{E_0},\]
  so
  \begin{equation}\label{D into E's G(r,1,n)}
    \csum{D_i}=\sum_{j=0}^{n} \phi_i(C_{n-j}) \csum{E_{n-j}}
  \end{equation}
  and 
  \begin{equation}\label{Phi inverse G(r,1,n)}
    \Phi^{-t}=\left(\text{Coeff.\ of $X^{n-j}$ in }\binom{n+\frac{X-1}{r}-i}{n}\right)_{0\leq i,j\leq n}.
  \end{equation}
\end{theorem}

Our combinatorial solution is an immediate corollary of Theorem~\ref{transition result G(r,1,n)}.
\begin{theorem}\label{G(r,1,n) combinatorial solution}
  For any fixed $z\in G_n$ with exactly $k$ descents,
  \begin{equation}\label{G(r,1,n) combinatorial solution Eq}
    [\phi_i\phi_j,\phi_k]=|\{(x,y)\in G_n\times G_n \mid \des(x)=i,\ \des(y)=j,\ xy=z\}|.
  \end{equation}
\end{theorem}
\begin{proof}
  As in the proof of Theorem~\ref{Sn combinatorial solution},
  the $\csum{D_i}$'s span a subalgebra of $\mathbb C[G_n]$, and the $\csum{E_i}$'s are orthogonal idempotents,
  so by
  \eqref{D into E's G(r,1,n)}, $[\phi_i\phi_j,\phi_k]$ is the coefficient of
  $\csum{D_k}$ in $\csum{D_i}\csum{D_j}$. 
  \end{proof}

  \subsection{} We end with $G_n$ analogues of
  three earlier results for $S_n$, namely, the formula of 
  Diaconis and Fulman in Theorem~\ref{Diaconis--Fulman formula},
  the useful inversion result in Theorem~\ref{General inversion},
  and the reformulation of Holte's Markov chain for adding random numbers
  in terms of products and Foulkes characters.
\subsubsection{}
Writing 
\[A_{r,n}=\sum_{x\in G_n} X^{\des(x)},\]
the 
analogue of \eqref{Eulerian diff} is 
\begin{equation}\label{G(r,1,n) Eulerian diff}
  \left[\left(1+Y\frac{d}{dY}\right)^n\frac{1}{1-Y^r}\right]_{Y=X^{1/r}}=\frac{A_{r,n}}{(1-X)^{n+1}},
\end{equation}
and the analogue of Theorem~\ref{Diaconis--Fulman formula} is the following,
with a version of the hyperoctahedral case already appearing in
earlier work of Diaconis and Fulman~\cite{DF2}.
\begin{theorem} For $0\leq j \leq n$, 
  \begin{equation}\label{r OP type formula}
    \sum_{i=0}^{n} \phi_i(C_j) X^i= \left[(1-Y^r)^{n+1}\left(1+Y\frac{d}{dY}\right)^j\frac{1}{1-Y^r}\right]_{Y=X^{1/r}}.
  \end{equation}
\end{theorem}

\begin{proof}
 The proof follows just as for $S_n$.  Denoting by $\phi_i^{(n)}$ and $C_j^{(n)}$ the $\phi_i$ and $C_j$ for  $G_n$,
  we have \cite[Theorem 7]{M1}
  \begin{equation}\label{G(r,1,n) row rec}
    \phi^{(n)}_i(C_j^{(n)})
    =
    \phi^{(n-1)}_i(C_j^{(n-1)})-\phi^{(n-1)}_{i-1}(C_j^{(n-1)})
  \end{equation}
  for $0\leq i\leq n$ and $0\leq j\leq n-1$, where we take $\phi_{-1}^{(n-1)}=\phi_{n}^{(n-1)}=0$. 
  So the $G_{n-1}$ cases of \eqref{r OP type formula} imply the first $n$ cases of \eqref{r OP type formula} for $G_n$, while
  equality holds for $j=n$ by 
  \eqref{G(r,1,n) Eulerian degrees}, which is Corollary 8.1 in \cite{M1},
  and \eqref{G(r,1,n) Eulerian diff}.
\end{proof}

\subsubsection{}
For the analogue of Theorem~\ref{General inversion}, let
\begin{equation}
  \phi=\phi_0+X\phi_1+X^2\phi_2+\ldots +X^n\phi_n.
\end{equation}

\begin{theorem}\label{r General inversion}
  For any two sequences $a_0,a_1,\ldots,a_n$ and $b_0,b_1,\ldots,b_n$, 
  \[\sum_{i=0}^n a_i X^i =\sum_{i=0}^n b_i \binom{n+\frac{X-1}{r}-i}{n}\]
  if and only if
  \[\sum_{i=0}^n a_i\phi(C_i)  =\sum_{i=0}^n b_i X^i.\]
\end{theorem}

\begin{proof}
  This follows from \eqref{Phi inverse G(r,1,n)} just as 
  Theorem~\ref{General inversion} followed from \eqref{Phi inverse}.
\end{proof}
\subsubsection{Carries in terms of products}
In \cite{DF2}, Diaconis and Fulman connected the hyperoctahedral Foulkes
characters with adding an even number $N$ of random numbers
using balanced digits and odd base $b$.
We shall denote the transition matrix
of the Diaconis--Fulman Markov chain by $M_B$, so 
\[M_B=(M_B(i,j))_{0\leq i,j\leq N}\]
with
\[M_B(i,j)=\text{chance}\left\{\text{next carry is }
    j-\frac{N}{2} \mid \text{ last carry is }i-\frac{N}{2}\right\}.\]
We rewrite this Markov chain
in terms of our inner product and products involving  
Foulkes characters of the hyperoctahedral group $B_N=G(2,1,N)$.

\begin{theorem}
  Let $\phi_0,\phi_1,\ldots,\phi_N$ be the Foulkes characters of
  $B_N$.  Then
  \[M_B(i,j)=[\phi_i,b^{\ld-N}\phi_j].\]
\end{theorem}

\begin{proof}
Denoting by $\Phi_B$ the Foulkes character table
$(\phi_i(C_{N-j}))_{0\leq i,j\leq N}$ for~$B_N$, 
Diaconis and Fulman showed that
\[\Phi_B^t M_B=D\Phi_B^t,\quad D={\rm diag}(b^0,b^{-1},\ldots, b^{-N}).\]

For $0\leq i,j\leq N$, let
\[
  \alpha_{ij}
  =
  \frac{1}{|B_N|}\sum_{x\in C_{N-j}}\frac{\epsilon_i(x)}{\epsilon_i(1)},
  \quad
  \epsilon_i
  =
  \chi_{((N-i),(1^i),\emptyset,\emptyset,\ldots,\emptyset)},
\]
and let
\[\Lambda=(\alpha_{ij})_{0\leq i,j\leq N}.\]
Then
\[
  \Lambda \Phi_B^t
  =
  \left(\left\langle \frac{\epsilon_i}{\epsilon_i(1)},\phi_j\right\rangle\right)_{0\leq i,j\leq N}
  =
  I.
\]
Hence
\[
  M_B
  =
  \Lambda D\Phi_B^t
  =
  \left(\left\langle
      \frac{\epsilon_i}{\epsilon_i(1)}b^{\ell-N},\phi_j
  \right\rangle\right)_{0\leq i,j\leq N}
  =
  \left([\phi_i,b^{\ell-N}\phi_j]\right)_{0\leq i,j\leq N}.
\]
\end{proof}

\subsection*{Acknowledgements}
The author would like to thank Jason Fulman and the referee
for several helpful comments.

\end{document}